\documentclass[11pt]{amsart}

\usepackage{amsmath}
\usepackage{amsfonts}
\usepackage{amssymb}
\newcommand{\R}{{\mathbb R}}
\newcommand{\N}{{\mathbb N}}

\usepackage{enumerate}
\usepackage{hyperref}
\usepackage{cleveref}
\usepackage{xcolor}

\usepackage{graphicx}
\usepackage{float}
\usepackage{subfigure}
\usepackage{tikz}

\newtheorem{theorem}{Theorem}[section]
\newtheorem{lemma}[theorem]{Lemma}

\newtheorem{remark}[theorem]{Remark}

\newtheorem{proposition}[theorem]{Proposition}

\newtheorem{example}[theorem]{Example}

\begin{document}

\title[Infinity Laplace equations on infinite graphs]{Some existence and uniqueness results for infinity Laplace equations on infinite graphs}

\author{Fengwen Han}
\address{Fengwen Han: School of Mathematics and Statistics, Henan University, 475004 Kaifeng, Henan, China}
\email{\href{mailto:fwhan@outlook.com}{fwhan@henu.edu.cn}}

\author{Tao Wang}
\address{Tao Wang: Beijing International Center for Mathematical Research, Peking University, 100871, Beijing, China}
\email{\href{mailto:taowang25@pku.edu.cn}{taowang25@pku.edu.cn}}

\date{Oct 31, 2025}

\begin{abstract}
    We study the Dirichlet problem of the following discrete infinity Laplace equation on unbounded subgraphs
    \begin{equation*}
        \Delta_{\infty}u(x):=\inf_{y\sim x}u(y)+\sup_{y\sim x}u(y)-2u(x)=f(x).
    \end{equation*}
  For the homogeneous case ($f=0$), the existence and uniqueness of sublinear solutions are established. This result is applied to prove the existence and uniqueness of sublinear solutions for the homogeneous (normalized) infinity Laplace equations on unbounded Euclidean domains. Uniqueness is also shown for the case $f \geq 0$ on trees.
\end{abstract}

\maketitle

\section{Introduction}

For a function $u \in C^2$ with $\nabla u(x) \ne 0$, the normalized infinity Laplacian on $\R^N$ is defined by
\begin{equation}
    \Delta_{\infty} u(x):=\frac{1}{|\nabla u(x)|^2}\sum_{i,j}u_{x_i}(x)u_{x_ix_j}(x)u_{x_j}(x).
\end{equation}
See \cite{CIL1992} for the definition in the viscosity sense. We study the existence and uniqueness of solutions to the following infinity Laplace equation 
\begin{align}\label{equ:main_continuous_equation}
    \begin{cases}
        \Delta_{\infty} u(x) = f(x), \quad & x\in \Omega;\\
        u(x) = g(x), & x\in \partial \Omega.
    \end{cases}
\end{align}
Due to the lack of regularity, viscosity theory is the only method to deal with this problem for a long time. We refer to \cite{O2005, ES2008, L2014} for regularity results.
For a bounded domain $\Omega \subset \R^N$, the unique viscosity solution $u$ with $f\equiv 0$ is exactly the \emph{absolutely minimal Lipschitz extension} of $g$, i.e. $\mathrm{Lip}_{U}u=\mathrm{Lip}_{\overline{U}}u$ for any open set $U\subset \Omega$. We refer to \cite{aronsson1967extension,jensen1993uniqueness,barles2001existence,aronsson2004tour} for more details. Lu and Wang \cite{lu2008pde} proved the existence and uniqueness of \eqref{equ:main_continuous_equation} in the case that $f >0$ or $f < 0$ on a bounded open subset of $\R^n$ using Perron's method. For (non-normalized) infinity Laplace equations, we refer to \cite{lu2008inhomogeneous, lindqvist2016notes} for more existence and uniqueness results. It is worth noting that the uniqueness result fails for equation \eqref{equ:main_continuous_equation} with  sign-changing $f$; see \cite{lu2008inhomogeneous} for a counterexample. We also refer to \cite{LY2012, G2004, LMZ2019} and the references therein for other topics on the infinity Laplacian.

Results for unbounded $\Omega$ are very few. The following result is due to Crandall, Gunnarsson, and Wang \cite{crandall2007uniqueness}.
\begin{theorem}[\cite{crandall2007uniqueness}]\label{CGW results}
    Let $\Omega \subset \mathbb{R}^N$ be bounded and $\partial \Omega$ be bounded. Let $u, v \in C(\overline{\Omega})$, and $\Delta_{\infty}u \geq 0$, $\Delta_{\infty}v \leq 0$ in $\Omega$. Assume also that 
    \[
     \limsup_{|x| \to \infty}\frac{u(x)}{|x|} \leq 0
    \]
    and 
    \[
     \liminf_{|x| \to \infty}\frac{v(x)}{|x|} \geq 0.
    \]
    Then for $x \in \Omega$, 
    \[
     u(x) - v(x) \leq \max_{\partial \Omega}(u - v).
    \]
\end{theorem}
Note that \Cref{CGW results} implies that when $f \equiv 0$, if equation \eqref{equ:main_continuous_equation} admits a sublinear solution, then the solution is unique. We specifically mention the more general comparison result on exterior domains in \cite{HZ2021}.

In 2009, Peres et al. \cite{PSSW2009} introduced the tug-of-war game, which is a two-player, zero-sum, stochastic game. Given a connected undirected graph $G = (V, E)$, where $V$ is the set of vertices and $E$ is the set of edges. For any $x, y \in V$, we write $x \sim y$ if there exists an edge connecting $x$ and $y$. The following discrete infinity Laplace equation on $G$ is intensively studied by a probabilistic method in \cite{PSSW2009}:
\begin{equation}
    \label{equ:DiscreteInfinityLaplaceEquation}
    \begin{cases}
        \Delta_{\infty}u(x)=f(x), \quad &x\in X \subset V,\\
        u(x)=g(x), & x \in Y = V\setminus X,
    \end{cases}
\end{equation}
where $f$ and $g$ are bounded functions on $X$ and $Y$ respectively, and
\begin{equation}
    \Delta_{\infty}u(x):=\inf\limits_{y\sim x}u(y)+\sup\limits_{y\sim x}u(y)-2u(x)
\end{equation}
is called the discrete infinity Laplacian (We use the same symbol as the normalized infinity Laplacian on $\R^N$). By a probabilistic method, for any graph, Peres et al. \cite{PSSW2009} proved the existence and uniqueness result for $f\equiv 0$, $\inf f>0$, or $\sup f <0$. The tug-of-war game presents a probabilistic interpretation to the equation \eqref{equ:main_continuous_equation}.

The $\varepsilon$-tug-of-war game introduced in \cite{PSSW2009} provides a discrete method to study the normalized infinity Laplace equation \eqref{equ:main_continuous_equation}. In fact, given a bounded domain $\Omega\subset \R^N$ and $\varepsilon>0$, a corresponding graph $G_{\varepsilon}=(V,E)$ is constructed via setting $V=\overline{\Omega}$, and $x\sim y$ if and only if $d_{\overline{\Omega}}(x,y) < \varepsilon$, where $d_{\overline{\Omega}}$ is the induced intrinsic metric of $\overline{\Omega}$. Then the solution of the following discrete infinity Laplace equation converges to a solution of \eqref{equ:main_continuous_equation} as $\varepsilon \to 0$
\begin{equation}
    \begin{cases}
        \Delta_{\infty}^{\varepsilon}u(x)=\varepsilon^2 f(x), \quad & x\in \Omega;\\
        u(x)=g(x), & x\in \partial \Omega,
    \end{cases}
\end{equation}
where $\Delta_{\infty}^{\varepsilon}$ defined via $$\Delta_{\infty}^{\varepsilon}u(x):=\inf\limits_{y\in B_x(\varepsilon)}u(y)+\sup\limits_{y\in B_x(\varepsilon)}u(y)-2u(x)$$ is the discrete infinity Laplacian on $G_{\varepsilon}$, $f \in C(\Omega)\cap L^{\infty}(\Omega)$, and $g\in C(\partial\Omega)$. The convergence was proved by \cite{PSSW2009} for $f\equiv 0$, $\inf f >0$, or $\sup f <0$ using a probabilistic method. Armstrong and Smart \cite{armstrong2012finite} introduced a ``boundary-biased'' $\varepsilon$-tug-of-war game, based on which they proved the convergence for all $f \in C(\Omega)\cap L^{\infty}(\Omega)$. We also refer to \cite{oberman2005convergent, le2007absolutely} for $f\equiv 0$, and \cite{peres2010biased} for other settings. 

Recently, Han and Wang \cite{han2025discrete} investigated the discrete infinity Laplace equation \eqref{equ:DiscreteInfinityLaplaceEquation} on a subgraph with finite width.
We say that a subgraph $X \subset V$ has finite width if the distances from all vertices to the boundary are uniformly bounded, i.e., $\mathrm{width}(X):= \sup\{d(x, V \setminus X): x \in X\}<\infty$, where $d(x, V \setminus X)$ is the combinatorial distance between $x$ and $V \setminus X$. Using Perron's method, they demonstrated the existence of bounded solutions. They also proved the uniqueness if $f\geq 0$ or $f\leq 0$ by establishing a comparison result.
\begin{theorem}[\cite{han2025discrete}]\label{thm:existence_and_uniquness_result_for_inhomogeneous_equations_on_bounded_width_graphs}
 Let $G=(V,E)$ be a graph, $X\subset V$ with $\mathrm{width}(X)<+\infty$, $f\in L^{\infty}(X)$ and $g \in L^{\infty}(V\setminus X)$. Then the discrete infinity Laplace equation \eqref{equ:DiscreteInfinityLaplaceEquation} admits a bounded solution. Moreover, the bounded solution is unique if $f \geq 0$ or $f\leq 0$.
\end{theorem}
\begin{theorem}[\cite{han2025discrete}]\label{thm:compare_result_for_inhomogeneous_equations_on_bounded_width_graphs}
    Let $G=(V,E)$ be a graph, $X\subset V$ with $\mathrm{width}(X)<+\infty$, $u,v\in C(V)$ be bounded and satisfy
    $$-\Delta_{\infty}u(x) \geq f(x) \geq -\Delta_{\infty}v(x), \ \forall \ x\in X,$$ 
    where $f$ is a nonnegative or nonpositive function on $X$.
    Then 
    \begin{equation}
        \sup\limits_{X}(u-v) \leq \sup\limits_{V \setminus X}(u-v).
    \end{equation}
\end{theorem}
By an argument of Arzel\`a–Ascoli, Han and Wang \cite{han2025discrete} proved that on Euclidean domains with finite width, the solutions of $\varepsilon$-tug-of-war games converge as $\varepsilon \to 0$. The result essentially establishes the existence of bounded solutions to normalized infinity Laplace equations on Euclidean domains with finite width.

In this paper, we proceed to study the existence and uniqueness of solutions to the discrete infinity Laplace equations. Given a graph $G = (V, E)$ with $V = U \sqcup \delta U$, where $\delta U$ is the boundary of $U$. We assume that $U$ has infinite width, i.e., there exists a sequence $\{x_n\} \subset U$ such that $d(x_n, \delta U) \to \infty$. Let $C(W)$ and $L^{\infty}(W)$ denote the spaces of functions and bounded functions on a subset $W \subset V$, respectively. Consider the following equation
\begin{align}\label{equ:main_discrete_equation}
    \begin{cases}
        \Delta_{\infty} u(x) = f(x), \quad & x\in U;\\
        u(x) = g(x), & x\in \delta U.
    \end{cases}
\end{align}

We first consider the homogeneous case and prove the existence and uniqueness of sublinear solutions to the equation \eqref{equ:main_discrete_equation}.
\begin{theorem}\label{thm:main_theorem_1}
    Let $g\in L^{\infty}(\delta U)$. The following equation
    \begin{align*}
        \begin{cases}
            -\Delta_{\infty}u(x)=0, \quad &x\in U;\\
            u(x) = g(x), &x\in \delta U;
        \end{cases}
    \end{align*}
    admits a unique solution satisfying
    \begin{align*}
        \limsup_{r \to \infty}\frac{\sup\limits_{d(y, \delta U) \leq r}|u(y)|}{r}=0.
    \end{align*}
    Moreover, the unique solution is bounded.
\end{theorem}

The existence of the solution is guaranteed by \Cref{thm:existence_and_uniquness_result_for_inhomogeneous_equations_on_bounded_width_graphs} and a exhaustion method, while its uniqueness follows directly from the following comparison result.

\begin{theorem}\label{thm:compare_result_for_sublinear_on_graph}
    Let $u,v \in C^{\infty}(V)$ satisfying:
    \begin{itemize}
        \item[(i)] $-\Delta_{\infty}v \geq 0 \geq -\Delta_{\infty}u$ on $U$;
        \item[(ii)] $\liminf\limits_{r \to \infty}\frac{\inf\limits_{d(y, \delta U)\leq r}v(y)}{r}\geq 0 \geq \limsup\limits_{r \to \infty}\frac{\sup\limits_{d(y,\delta U)\leq r}u(y)}{r}$.
    \end{itemize} 
    Then 
    \begin{equation}
        \sup\limits_{U}(u-v) \leq \sup\limits_{\delta U}(u-v).
    \end{equation}
\end{theorem}

As an application of the two theorems above, we prove the existence and uniqueness of sublinear solutions to the homogeneous (normalized) infinity Laplace equations on unbounded Euclidean domains.

\begin{theorem}\label{thm:comparison_result_of_continuous_case}
    Let $\Omega \subset \R^N$ be an unbounded domain with boundary $\partial \Omega$, $u,-v$ be infinity subharmonic on $\Omega$, uniformly continuous and bounded on $\partial \Omega$, and satisfy 
    \begin{equation*}
        \limsup\limits_{r\to \infty}\frac{\sup\limits_{d_{\overline{\Omega}}(y, \partial \Omega) \leq r}u(y)}{r}\leq 0 \leq \liminf\limits_{r\to \infty}\frac{\inf\limits_{d_{\overline{\Omega}}(y, \partial \Omega) \leq r}v(y)}{r}.
    \end{equation*}
    Then 
    \begin{equation}
        u(x)-v(x) \leq \sup_{\partial \Omega}(u-v), \ \forall \ x\in \Omega.
    \end{equation}
\end{theorem}

\begin{theorem}\label{thm:existence_result_of_continuous_case}
    Let $\Omega \subset \R^N$ be an unbounded domain with boundary $\partial \Omega$, $g$ be a bounded Lipschitz function on $\partial \Omega$. Then the equation 
    \begin{align}\label{equ:homogeneous_equation_on_unbounded_domain}
        \begin{cases}
            -\Delta_{\infty}u(x)=0, \quad &x\in \Omega;\\
            u(x) = g(x), &x\in \partial \Omega;
        \end{cases}
    \end{align}
    admits a unique uniformly continuous solution satisfying 
    \begin{align*}
        \limsup\limits_{r\to \infty}\frac{\sup\limits_{d_{\overline{\Omega}}(y, \partial\Omega) \leq  r}|u(y)|}{r}=0.
    \end{align*}
\end{theorem}
\begin{remark}
    Several known results can be derived as corollaries of \Cref{thm:comparison_result_of_continuous_case} and \Cref{thm:existence_result_of_continuous_case}. See \cite{crandall2007uniqueness} Theorem 3.2, 4.1 and 4.2.
\end{remark}

We then consider inhomogeneous equations on trees with vertex set $U \sqcup \delta U$, where $\delta U$ serves as the boundary. We say that a tree has bounded boundary if $\sup\limits_{x, y \in \delta U}d(x, y) < +\infty$. Assume that $f \geq 0$, then we have the following uniqueness theorem.
\begin{theorem}\label{thm:inhomogeneous_equations_general_case}
    Let $T$ be a tree with bounded boundary $\delta U$, $g\in L^{\infty}(\delta U)$, and let $f\in C(U)$ be nonnegetive. Then the solution to the following equation
    \begin{align*}
        \begin{cases}
            \Delta_{\infty}u(x)=f(x), \ &x\in U = V\setminus \delta U;\\
            u(x)=g(x), &x \in \delta U;
        \end{cases}
    \end{align*} 
    that satisfies
    \begin{equation*}
        \limsup\limits_{r\to \infty}\frac{\sup\limits_{y \in B_r(\delta U)}|u(y)|}{r}=0
    \end{equation*}
    is unique.
\end{theorem}

We provide a necessary condition for the existence of a sublinear solution. See Remark \ref{rmk:necessray_condition_for_sublinear_solution} for details.

The rest of this paper is organized as follows. In Section \ref{sec2}, we introduce some basic notions and definitions. In particular, we provide two equivalent characterizations for infinity subharmonic functions. In Section \ref{sec3}, we discuss homogeneous equations and prove \Cref{thm:main_theorem_1} and \Cref{thm:compare_result_for_sublinear_on_graph}. By combining these theorems and some known results, we then prove \Cref{thm:comparison_result_of_continuous_case} and \Cref{thm:existence_result_of_continuous_case}. In Section \ref{sec4}, we discuss inhomogeneous equations on trees and prove \Cref{thm:inhomogeneous_equations_general_case}.

\section{Preliminaries}\label{sec2}
Let $G = (V, E)$ be a graph. For any $x, y \in V$, we define the \textit{combinatorial distance} between $x$ and $y$ by 
\[
 d(x, y):= \inf\{n: x = x_0 \sim x_1 \sim \cdots \sim x_n = y\},
\]
that is, the length of a shortest path connecting $x$ and $y$. For any $x \in V$, we write the following:
\begin{itemize}
 \item $B_r(x):= \{y \in V: d(x, y) \leq r\}$, which is called the closed $r$-ball centered at $x$;
 \item $S_r(x):= \{y \in V: d(x, y) = r\}$, which is called the $r$-sphere centered at $x$.
\end{itemize}
For any subset $U \subseteq V$, denote the distance between $x$ and $U$ by 
\[
 d(x, U):= \inf\{n: x = x_0 \sim x_1 \sim \cdots \sim x_n \in U\}.
\]
We define the boundary of $U$ as 
\[
 \delta U:= \{y \notin U: \text{ there exists } x \in U \text{ such that } y \sim x\}.
\]
We write $\overline{U}:= U \cup \delta U$. We denote by $C(U)$ the set of functions on $U$. For any function $u$, set
    $$S^+u(x):=\sup_{y\sim x}(u(y)-u(x)),\ S^-u(x):=\sup_{y\sim x}(u(x)-u(y)),$$
and define $$L(u,x):=\max\{|S^+u(x)|,|S^-u(x)|\}=\sup_{y\sim x}|u(x)-u(y)|.$$

Given a function $u \in C(V)$, the discrete infinity Laplacian $\Delta_{\infty}$ is defined as 
$$\Delta_{\infty}u(x):=S^{+}u(x)-S^{-}u(x)=\sup_{y\sim x}u(y)+\inf_{y\sim x}u(y)-2u(x).$$
We say $u \in C(V)$ is infinity subharmonic if $\Delta_{\infty}u \geq 0$ on $V$, $u$ is infinity superharmonic if $\Delta_{\infty}u \leq 0$, and $u$ is infinity harmonic if $u$ is both infinity subharmonic and superharmonic. 
\begin{proposition}\label{prop:conditions_for_subharmonic}
 Let $U \subset V$ be a connected subset, $u \in C(\overline{U})$. The following conditions are equivalent:
 \begin{itemize}
  \item[(i)] $u$ is infinity subharmonic on $U$.
  \item[(ii)] For any $x \in U$, $L(u, x) = \sup\limits_{y \sim x}u(y) - u(x)$.
  \item[(iii)] For any $x \in U$ and any $r \in \mathbb{N}_+$ with $r \leq d(x, \delta U)$,
   \[
    L(u, x) \leq \frac{\sup\limits_{z \in B_r(x)}u(z) - u(x)}{r}.
   \]
 \end{itemize}
\end{proposition}
\begin{proof}
 We firstly prove that the equivalence between the condition (i) and (ii). Suppose that condition (i) holds, then for any $x \in U$, $\Delta_{\infty}u(x) \geq 0$ implies that 
 \[
  \sup_{y \sim x}u(y) - u(x) \geq 0,
 \]
 and
 \[
  \sup_{y \sim x}u(y) - u(x) \geq u(x) - \inf_{y \sim x}u(y) \geq u(x) - \sup_{y \sim x}u(y).
 \]
 Thus,
 \[
  L(u, x) = \sup_{y \sim x}|u(y) - u(x)| = \sup_{y \sim x}u(y) - u(x),
 \]
 i.e., condition (ii) holds. 
 
 Now suppose that condition (ii) holds, then by the definition, 
 \[
  |\inf_{y \sim x}u(y) - u(x)| \leq \sup_{y \sim x}u(y) - u(x),
 \]
 and thus, $\Delta_{\infty}u(x) \geq 0$, i.e., condition (i) holds.
 
 Note that condition (iii) clearly implies condition(ii) by setting $r = 1$. Thus, to complete the proof, it suffices to show that condition (i) implies condition (iii).
 
 Suppose that $u$ is infinity subharmonic on $U$. For any $x \in U$, if $L(u, x) = 0$ or $L(u, x) = +\infty$, then it is easy to check that 
   \[
    L(u, x) \leq \frac{\sup\limits_{z \in B_r(x)}u(z) - u(x)}{r}, \ \text{ for any } 1 \leq r \leq d(x, \delta U).
   \]
  In the following, we assume that $+\infty > L(u, x) > 0$. For any $0 < \varepsilon < \frac{L(u, x)}{2}$, consider a path $P: x= x_0 \sim x_1 \sim \cdots \sim x_r$ satisfying 
  \[
   u(x_{i+1}) \geq \sup_{y \sim x_i}u(y) - \frac{\varepsilon}{2^i}, \ \forall \ 0 \leq i \leq r-1.
  \]
  Then we have 
  \[
   L(u, x) = \sup_{y \sim x_0}u(y) - u(x_0) \leq u(x_1) - u(x_0) + \varepsilon.
  \]
  Moreover, $\Delta_{\infty}u(x_i) \geq 0$ implies that
  \begin{align*}
   u(x_{i+1}) - u(x_i) &\geq u(x_i) - u(x_{i-1}) - \frac{\varepsilon}{2^i} \\
   &\geq u(x_{i-1}) - u(x_{i-2}) - \left(\frac{\varepsilon}{2^{i}} + \frac{\varepsilon}{2^{i-1}}\right) \\
   &\mathrel{\phantom{=}}\vdots \\
   &\geq u(x_1) - u(x_0) -  \left(\frac{\varepsilon}{2^{i}} + \frac{\varepsilon}{2^{i-1}} + \cdots +\frac{\varepsilon}{2}\right) \\
   &\geq L(u, x) -  \left(\frac{\varepsilon}{2^{i}} + \frac{\varepsilon}{2^{i-1}} + \cdots + \frac{\varepsilon}{2} + \varepsilon \right) \\
   &\geq L(u, x) - 2\varepsilon.
  \end{align*}
  It follows that 
  \[
   u(x_r) - u(x_0) \geq rL(u, x) - 2r\varepsilon.
  \]
  Note that $x_r \in B_r(x)$ and thus
  \[
   \frac{\sup\limits_{z \in B_r(x)}u(z) - u(x)}{r} \geq \frac{u(x_r) - u(x_0)}{r} \geq L(u, x) - 2\varepsilon.
  \]
  By letting $\varepsilon \to 0$, we get condition (iii).
\end{proof}

\begin{remark}
 Unlike the continuous case, in the discrete case, the condition that $u$ is infinity subharmonic on $U$ is not equivalent to the following condition:
   \begin{equation}\label{equ:counterexamplecondition}
    L(u, x) \leq \frac{\sup\limits_{z \in S_r(x)}u(z) - u(x)}{r}, \ \forall \ x \in U \text{ and } 1 \leq r \leq d(x, \delta U).
   \end{equation}
 See the following example.
\end{remark}

\begin{example}
 Consider the graph shown in Figure \ref{fig:counterexample}, where $$V = \{x, x_1, x_2, \cdots, x_n, \cdots, y_1, y_2, \cdots, y_n, \cdots\}.$$ The edges consist forms of $x \sim x_i$, $x_i \sim y_i$ and $x_i \sim x_{2i}$ for all $i \geq 1$. Let $U = \{x, x_1, x_2, \cdots, x_n, \cdots\}$ and thus, $\delta U = \{y_1, y_2, \cdots, y_n, \cdots\}$. Define function $u \in C(V)$ via $u(x) = 0$, $u(x_i) = 2i$ and $u(y_i) \equiv 1$. It is easy to check that $u$ is infinity subharmonic on $U$, $d(x, \delta U) = 2$ and $L(u, x) = \infty$. However, by letting $r = 2$, we have
 \[
  \frac{\sup\limits_{z \in S_2(x)}u(z) - u(x)}{2} = \frac{1}{2} < L(u, x).
 \]
      \begin{figure}[H]
        \centering
        \begin{tikzpicture}[scale=2]
            \draw[thick, gray] (0,0) -- (2,0) (0,0) -- (1.4,1.4) (0,0) -- (1.4,-1.4) (0,0) -- (0,-2) (0,0) -- (-2,0) (0.7,0.7) -- (1,0);
            \draw[gray, thick, dashed] (0,-1) arc(-90:-315:1);
            \draw[gray, thick] (1,0) .. controls (1.4,-1.4) .. (0,-1);
            \draw[gray, thick, dashed] (0.7,-0.7) .. controls (0,-2) .. (-0.7,-0.7);
            \node at (-0.3, 0.3) {$x$}; \node at (0.7,1.0) {$x_1$}; \node at (1.7,1.4) {$y_1$};
            \node at (1.3,0.3) {$x_2$}; \node at (2.3,0) {$y_2$}; \node at (0.7, -1.0) {$x_3$};
            \node at (1.7,-1.4) {$y_3$}; \node at (-0.3,-0.7) {$x_4$}; \node at (0, -2.3) {$y_4$};
            \node at (-1.3,0.3) {$x_n$}; \node at (-2.3, 0) {$y_n$};
            \foreach \a in {0,1,2}
                \filldraw (\a,0) circle(2pt) (0,-\a) circle(2pt) (-\a, 0) circle(2pt);
            \foreach \b in {0.7,1.4}
                \filldraw (\b,\b) circle(2pt) (\b,-\b) circle(2pt);
        \end{tikzpicture}
        \caption{Infinity subharmonic is not equivalent to the condition \eqref{equ:counterexamplecondition}.}
        \label{fig:counterexample}
    \end{figure}
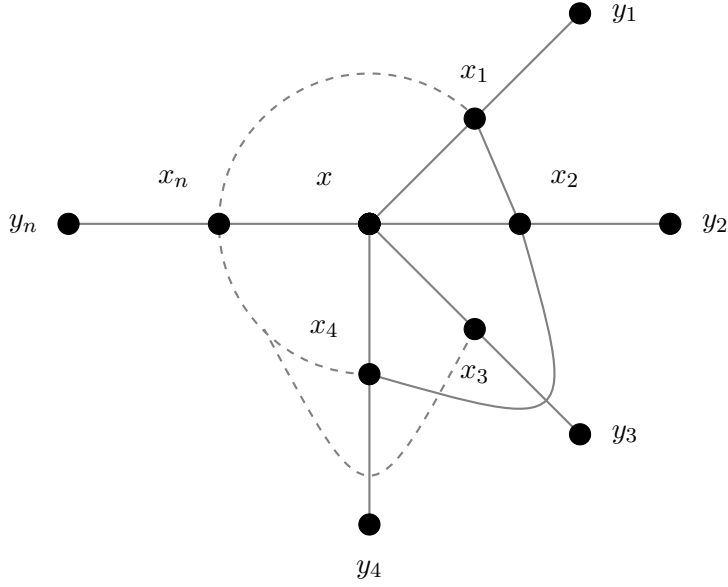
\end{example}
It is worth noting that if $u$ is bounded, then $\Delta_{\infty}u \geq 0$ on $U$ and condition \eqref{equ:counterexamplecondition} are equivalent by \Cref{thm:compare_result_for_inhomogeneous_equations_on_bounded_width_graphs}.

\section{For homogeneous equations}\label{sec3}
In this section, we study homogeneous equations on an unbounded graph $G = (V, E)$, where $V = U \sqcup \delta U$. For simplicity, we write 
 \[
  |x|:= d(x, \delta U),
 \]
 and for any $r \geq 1$,
 \[
  B_r(\delta U):= \{y \in U: |y| \leq r\};
 \]
 \[
  S_r(\delta U):= \{y \in U: |y| = r\}.
 \]
Let $u \in C(V)$ be infinity subharmonic on $U$. For any $\varepsilon > 0$, we firstly construct a function $u_{\varepsilon} \in C(V)$ that is infinity subharmonic and $L(u_{\varepsilon}, x) \geq \varepsilon$ on $U$. For this purpose, let $O_{\varepsilon} = \{x \in U: L(u, x) \leq \varepsilon\}$. Then for each connected component $K$ of $O_{\varepsilon}$, we define
\[
 w_{\varepsilon}(x):= \sup_{y \in \delta K}(u(y) - \varepsilon d_K(x, y)), \quad x \in K,
\]
where $d_K$ is the induced intrinsic metric of $K$, i.e., for any $x, y \in \overline{K}$,
\[
 d_K(x, y):=\inf\{n: x=x_0\sim x_1\sim \cdots \sim x_n = y, \text{with } x_i \in K, 1 \leq i \leq n-1\}.
\]
Consider the following functions
\begin{equation*}
 u_{\varepsilon}(x):= 
  \begin{cases}
   u(x), \quad &x \in V \setminus O_{\varepsilon}, \\
   w_{\varepsilon}(x), &x \in O_{\varepsilon}.
  \end{cases}
\end{equation*}

\begin{proposition}\label{prop:uvarepsilon}
    The functions $\{u_{\varepsilon}\}_{\varepsilon > 0}$ have the following properties:
    \begin{itemize}
        \item[(i)] $u_{\varepsilon}$ is infinity subharmonic on $U$.
        \item[(ii)] $u_{\varepsilon}=u$ on $V \setminus O_{\varepsilon}$, and $u_{\varepsilon}\leq u$ on $V$.
        \item[(iii)] $L(u_{\varepsilon},x) \geq \varepsilon$ for $x\in U$.
        \item[(iv)] $\lim\limits_{\varepsilon\to 0^+}u_{\varepsilon}(x)=u(x)$ for all $x\in U$.
    \end{itemize}
\end{proposition}
\begin{proof}
 We firstly establish (ii). By definition, $u_{\varepsilon} = u$ on $V \setminus O_{\varepsilon}$. To show that $u_{\varepsilon} \leq u$ on $O_{\varepsilon}$, we claim that for any $x \in K$, $y \in \overline{K}$,
\[
 |u(x) - u(y)| \leq \varepsilon d_K(x, y),
\]
where $K$ is the connected component of $O_{\varepsilon}$ containing $x$. Indeed, for any path $P: x = x_0 \sim \cdots \sim x_n = y$ with $x_i \in K$, $i \leq n-1$, since $L(u, x_i) \leq \varepsilon$, we have
\[
 |u(x_i) - u(x_{i+1})| \leq \varepsilon, \quad i \leq n-1.
\]
It follows that
\[
 |u(x) - u(y)| \leq |u(x_0) - u(x_1)| + \cdots + |u(x_{n-1}) - u(x_n)| \leq n\varepsilon.
\]
This proves the claim. Thus, for any $x \in K$, $y \in \overline{K}$,
\[
 u(y) - \varepsilon d_K(x, y) \leq u(x),
\]
and hence, $w_{\varepsilon}(x) = u_{\varepsilon}(x) \leq u(x)$.

We now show that $u_{\varepsilon}$ is infinity harmonic on $U$. It suffices to prove that for any $x \in U$, $\Delta_{\infty}u_{\varepsilon}(x) \geq 0$. If $x$ is in the interior of $V \setminus O_{\varepsilon}$, that is, for any $y \sim x$, $y \notin O_{\varepsilon}$, then since $u_{\varepsilon} = u$ on $V \setminus O_{\varepsilon}$ and $u$ is infinity subharmonic on $U$, we have $\Delta_{\infty}u_{\varepsilon}(x) \geq 0$.

If $x \in \delta O_{\varepsilon}$, i.e., $x \notin O_{\varepsilon}$ and there exists $y \in O_{\varepsilon}$ such that $y \sim x$. Note that $u$ is infinity subharmonic implies that $L(u, x) > \varepsilon$. By \Cref{prop:conditions_for_subharmonic}, there exists $z \sim x$ such that $u(z) - u(x) > \varepsilon$. This implies that $z \notin O_{\varepsilon}$. Therefore, we have 
\begin{align*}
 \sup_{z \sim x}u_{\varepsilon}(z) - u_{\varepsilon}(x) &\geq \sup_{\substack{z \sim x \\ z \notin O_{\varepsilon}}}u_{\varepsilon}(z) - u(x) \\
 &= \sup_{\substack{z \sim x \\ z \notin O_{\varepsilon}}}u(z) - u(x) = \sup_{z \sim x}u(z) - u(x) > \varepsilon.
\end{align*}
On the other hand, 
\begin{align*}
 \inf_{z \sim x}u_{\varepsilon}(z) - u_{\varepsilon}(x) &= \min\left\{\inf_{\substack{z \sim x \\ z \in O_{\varepsilon}}}u_{\varepsilon}(z) - u_{\varepsilon}(x), \inf_{\substack{z \sim x \\ z \notin O_{\varepsilon}}}u_{\varepsilon}(z) - u_{\varepsilon}(x)\right\} \\
 &\geq \min\left\{-\varepsilon, \inf_{z \sim x}u(z) - u(x)\right\},
\end{align*}
and thus, 
\[
 \Delta_{\infty}u_{\varepsilon}(x) \geq \min\left\{0, \Delta_{\infty}u(x)\right\} \geq 0.
\]

If $x \in O_{\varepsilon}$, let $K$ be the connected component of $O_{\varepsilon}$ containing $x$. Since for any $z \in \delta K$, any path connecting $x$ and $z$ must pass through $S_1(x)$, there exists $x_z \in S_1(x)$ such that
\[
 d_K(x, z) = 1 + d_K(x_z, z) \geq 1 + \inf_{y \in S_1(x)}d_K(y, z),
\]
and thus,
\begin{align*}
 u(z) - \varepsilon d_K(x, z) &\leq \sup_{y \in S_1(x)}\left(u(z) - \varepsilon d_K(y, z)\right) - \varepsilon \\
 &\leq \max\left\{\sup_{\substack{y \in S_1(x) \\ y \in \delta K}}u(y), \sup_{\substack{y \in S_1(x) \\ y \in O_{\varepsilon}}}\left(u(z) - \varepsilon d_K(y, z)\right)\right\}-\varepsilon.
\end{align*}
It follows that 
\[
 u_{\varepsilon}(x) \leq \sup_{y \sim x}u_{\varepsilon}(y) - \varepsilon.
\]
On the other hand, for any $y \sim x$, if $y \in O_{\varepsilon}$, then since $d_K(y, z) \leq d_K(x, z) + 1$ for any $z \in \delta K$, we have $u_{\varepsilon}(y) \geq u_{\varepsilon}(x) - \varepsilon$. If $y \notin O_{\varepsilon}$, then since $L(u, x) \leq \varepsilon$, we have $u_{\varepsilon}(y) = u(y) \geq u(x) - \varepsilon \geq u_{\varepsilon}(x) - \varepsilon$. Thus, we always have that
\[
 \inf_{y \sim x}u_{\varepsilon}(y) - u_{\varepsilon}(x) \geq -\varepsilon,
\]
which implies that $\Delta_{\infty}u_{\varepsilon}(x) \geq 0$. We complete the proof of (i).

Note that the proof of (i) also implies $L(u_{\varepsilon}, x) \geq \varepsilon$ for $x \in U$, i.e., (iii) holds.

We finally prove (iv). For any $x \in U$, if $L(u, x) > 0$, then it is obvious that (iv) holds. In the following we assume that $L(u, x) = 0$. Let $\mathcal{N}_x$ be the connected component of $\{y \in V: u(y) = u(x)\}$. Since $L(u ,x) = 0$, we have $B_1(x) \subset \mathcal{N}_x$. Define
\[
 r_x := \min\{r: B_r(x) \subset \mathcal{N}_x, \exists \ y \in S_r(x) \text{ such that } y \in \delta U \text{ or } L(u, y) > 0\}.
\]
Note that $1 \leq r_x \leq |x|$. For such a $y \in B_{r_x}(x)$, if $y \in \delta U$, then for any $\varepsilon > 0$, otherwise for $\varepsilon < L(u, y)$, we have
\[
 u(x) \geq u_{\varepsilon}(x) \geq u(y) - \varepsilon r_x = u(x) - \varepsilon r_x \geq u(x) - \varepsilon |x| \to u(x), \ \text{as } \varepsilon \to 0.
\]
We complete the proof.
\end{proof}

\begin{remark}
 The definition of $O_{\varepsilon}$ cannot be changed to $\{x \in U: L(u, x) < \varepsilon\}$. If we were to define $O_{\varepsilon} = \{x \in U: L(u, x) < \varepsilon\}$, then $u_{\varepsilon}$ might not be infinity subharmonic on $U$. See the following example.
\end{remark}

\begin{example}
 Consider the graph shown in Figure \ref{fig:counterexample1}. Let $U = V \setminus \{\bar{x}\}$, and thus, $\delta U = \{\bar{x}\}$. Define the function $u$ via 
 \[
  u(x) = u(\bar{x}) = 0,
 \]
 \[ 
  u(y_i) = \varepsilon - \frac{\varepsilon}{2^i}, \quad \forall \ i \geq 1,
 \]
 \[
  u(z_i) = -\varepsilon + \frac{\varepsilon}{2^i}, \quad \forall \ i \geq 1,
 \] 
and then extend it linearly to the whole graph. It is easy to check that $u$ is infinity subharmonic on $U$. If we were to define $O_{\varepsilon} = \{w \in U: L(u, w) < \varepsilon\}$, then $O_{\varepsilon} = V \setminus \{x, \bar{x}\}$ and $\{x\} = \delta O_{\varepsilon}$. It follows that $u_{\varepsilon}(x) = u_{\varepsilon}(\bar{x}) = 0$, $u_{\varepsilon}(y_i) = u_{\varepsilon}(z_i) = -\varepsilon$, and thus, $\Delta_{\infty}u_{\varepsilon}(x) = -\varepsilon < 0$.
    \begin{figure}
        \centering
        \begin{tikzpicture}[scale=2]
            \draw[thick, gray] (-1,0) -- (2,0) (-0.7,-0.7) -- (1.4,1.4) (-0.7,0.7) -- (1.4,-1.4) (0,0) -- (0,2);
            \draw[gray, thick, dashed] (1.4,1.4) -- (2.5,2.5) (2,0) -- (3.5,0) (1.4,-1.4) -- (2.5,-2.5);
            \draw[gray, thick, dashed] (-0.5,0.85) -- (1.5, -2.55);
            \draw[gray, thick, dashed] (2.5,0) arc(-0:-90:2.5)  (-0.8,0) arc(-180:-260:0.8);
            \filldraw (0,2) circle(2pt);
            \node at (0, -0.3) {$x$}; \node at (0,2.3) {$\bar{x}$}; \node at (0.7,0.4) {$y_1$}; \node at (1.2,0.3) {$y_2$};
            \node at (1,-0.7) {$y_n$}; \node at (-0.7,-1) {$z_1$}; \node at (-1,-0.3) {$z_2$}; \node at (-1,0.7) {$z_n$};
            \foreach \a in {-1,0,1,2,3}
                \filldraw (\a,0) circle(2pt);
            \foreach \b in {-0.7,0.7,1.4,2.1}
                \filldraw (\b,\b) circle(2pt) (\b,-\b) circle(2pt);
        \end{tikzpicture}
        \caption{The definition of $O_{\varepsilon}$ cannot be changed to $\{x \in U: L(u, x) < \varepsilon\}$.}
        \label{fig:counterexample1}
    \end{figure}
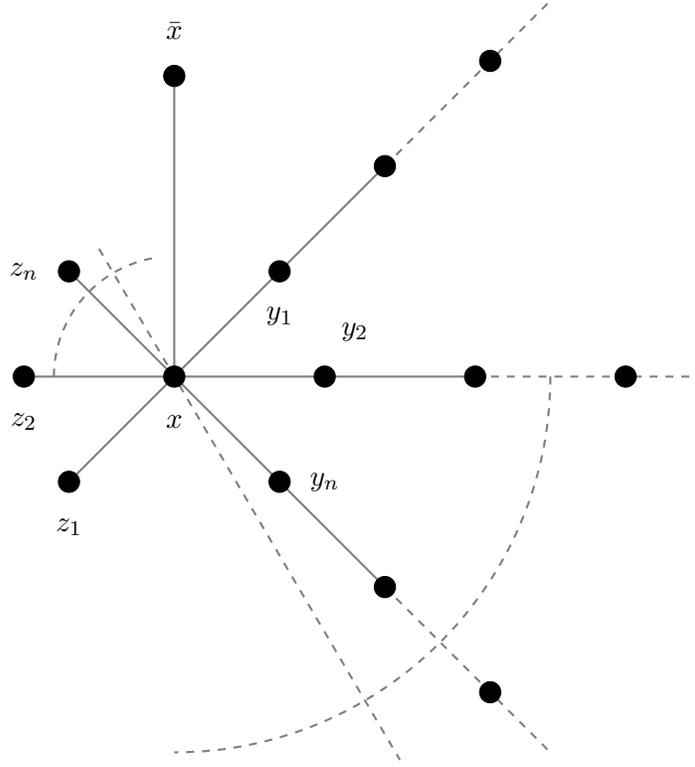
\end{example}

\begin{lemma}\label{lem:compare_for_uvarepsilon}
    Let $u\in C(V)$ be infinity subharmonic on $U$, $\varepsilon >0$ and $u_{\varepsilon}$ be defined as above. Suppose that
    \begin{equation*}
        \limsup_{r \to \infty}\frac{\sup\limits_{y \in B_r(\delta U)}u(x)}{r}\leq 0,
    \end{equation*}
    then 
    \begin{equation*}
        u_{\varepsilon}(x) \leq \sup\limits_{\delta U}u - \varepsilon |x|.
    \end{equation*}
\end{lemma}
\begin{proof}
    Let $\varepsilon_0 < \varepsilon$. For any $x_0\in U$, consider an infinite path $P: x_0\sim x_1 \sim x_2 \sim \cdots$ satisfying
    $$u_{\varepsilon}(x_i) \geq \sup\limits_{y\sim x_{i-1}}u_{\varepsilon}(y)-\varepsilon_0.$$
    We assert that there exists a $K$ such that $x_K\in \delta U$. Otherwise, note that $u_{\varepsilon}$ is infinity subharmonic, then by Proposition \ref{prop:conditions_for_subharmonic} and Proposition \ref{prop:uvarepsilon}, we have
  \[
        \sup_{y\sim x_{i-1}}u_{\varepsilon}(y)-u_{\varepsilon}(x_{i-1}) = L(u_{\varepsilon},x_{i-1}) \geq \varepsilon,
  \]
  which implies that 
  \[
    u_{\varepsilon}(x_{i})-u_{\varepsilon}(x_{i-1}) \geq \varepsilon - \varepsilon_0.
  \]
    Then we have
    \begin{align}\label{eq:inequality_for_marching_path}
        u_{\varepsilon}(x_{i}) \geq (\varepsilon - \varepsilon_0)i + u_{\varepsilon}(x_0).
    \end{align}
    It follows that 
    \begin{align*}
       0 \geq \limsup_{r \to \infty}\frac{\sup\limits_{y \in B_r(\delta U)}u(y)}{r} &\geq \limsup_{r \to \infty}\frac{\sup\limits_{y \in B_r(\delta U)}u_{\varepsilon}(y)}{r} \\
       &\geq \limsup_{r \to \infty}\frac{\sup\limits_{|x_i| \leq r+|x_0|}u_{\varepsilon}(x_i)}{r+|x_0|} \\
       &\geq \limsup_{r \to \infty}\frac{u_{\varepsilon}(x_r)}{r+|x_0|} \\
       &\geq \varepsilon - \varepsilon_0 >0,
    \end{align*}
    where the penultimate inequality is due to the fact that $|x_r| \leq r + |x_0|$. The contradiction shows that the claim holds. Let $K_0$ be the first index such that $x_{K_0}\in \delta U$. Note that \eqref{eq:inequality_for_marching_path} holds for $i \le K_0$, we have 
    $$u_{\varepsilon}(x_{K_0}) \geq (\varepsilon - \varepsilon_0)K_0 + u_{\varepsilon}(x_0).$$
    Since $K_0 \geq |x_0|$ and $u_{\varepsilon}(x_{K_0}) \leq \sup\limits_{\delta U}u$, we have
    \begin{align*}
        u_{\varepsilon}(x_0) \leq \sup\limits_{\delta U}u - (\varepsilon-\varepsilon_0) |x_0|.
    \end{align*}
    By letting $\varepsilon_0 \to 0$, we get the result.
\end{proof}

We now prove \Cref{thm:compare_result_for_sublinear_on_graph}.
\begin{proof}[Proof of \Cref{thm:compare_result_for_sublinear_on_graph}]
    Let $\varepsilon > 0$, it suffices to establish the result for $u_{\varepsilon}$. By Lemma \ref{lem:compare_for_uvarepsilon}, there exists a $r_{\varepsilon}>0$ such that
    \begin{align*}
        u_{\varepsilon}(x) \leq v(x), \ \forall \ |x| \geq r_{\varepsilon}.
    \end{align*}
    Then the result follows from \Cref{thm:compare_result_for_inhomogeneous_equations_on_bounded_width_graphs}.
\end{proof}

With the help of \Cref{thm:compare_result_for_sublinear_on_graph}, we complete the proof of \Cref{thm:main_theorem_1}.
\begin{proof}[Proof of \Cref{thm:main_theorem_1}]
    The uniqueness follows from \Cref{thm:compare_result_for_sublinear_on_graph}. To complete the proof, it suffices to show the existence of the solution. For any $r\in \N^{+}$, by \Cref{thm:existence_and_uniquness_result_for_inhomogeneous_equations_on_bounded_width_graphs}, we know that the following equation admits a unique bounded solution, denoted by $u^{r}$: 
    \begin{align*}
        \begin{cases}
            -\Delta_{\infty}u^r(x)=0, \quad &x\in B_r(\delta U);\\
            u^{r}(x) = g(x) &x \in \delta U;\\
            u^{r}(x) = \sup g, &x \in S_{r+1}(\delta U).
        \end{cases}
    \end{align*}
    Then $u(x)=\lim\limits_{r\to\infty}u^{r}(x)$ is the bounded solution we want.
\end{proof}

We now consider the continuous case. Let $\Omega \subset \R^N$ be a domain with boundary $\partial \Omega$. For any $\varepsilon > 0$, we define 
\[
 \Omega_{\varepsilon}:=\{x \in \Omega: d_{\overline{\Omega}}(x, \partial \Omega) > \varepsilon\},
\]
where $d_{\overline{\Omega}}$ is the induced intrinsic metric of $\overline{\Omega}$. Given a continuous function $u \in C(\overline{\Omega})$, we use the notation 
\[
 u^{\varepsilon}(x):= \max_{\overline{B}_{\varepsilon}(x)}u \quad \text{and } \quad u_{\varepsilon}(x):= \min_{\overline{B}_{\varepsilon}(x)}u, \ x \in \Omega_{\varepsilon}.
\]
Recall that discrete infinity Laplacian $\Delta_{\infty}^{\varepsilon}$ on graph $G_{\varepsilon}= (\overline{\Omega}, E)$ is defined via
\[
 \Delta_{\infty}^{\varepsilon}u(x):=\inf\limits_{y\in B_x(\varepsilon)}u(y)+\sup\limits_{y\in B_x(\varepsilon)}u(y)-2u(x).
\]
\begin{lemma}[\cite{armstrong2010easy}]\label{lem:armstrong_result}
    If $u$ is infinity subharmonic on $\Omega$, then 
    \begin{equation}
        \Delta_{\infty}^{\varepsilon}u^{\varepsilon}(x) \geq 0, \ \forall \ x\in \Omega_{2\varepsilon},
    \end{equation}
    and if $v$ is infinity superharmonic on $\Omega$, then 
    \begin{equation}
        \Delta_{\infty}^{\varepsilon}v_{\varepsilon}(x) \leq 0, \ \forall \ x\in \Omega_{2\varepsilon}.
    \end{equation}
\end{lemma}

\begin{proof}[Proof of \Cref{thm:comparison_result_of_continuous_case}]
 By \Cref{lem:armstrong_result} and \Cref{thm:compare_result_for_sublinear_on_graph}, for any $x \in \Omega_{2\varepsilon}$, we have
 \[
  u^{\varepsilon}(x) - v_{\varepsilon}(x) \leq \sup_{\Omega_{\varepsilon} \setminus \Omega_{2\varepsilon}}(u^{\varepsilon} - v_{\varepsilon}).
 \]
 By letting $\varepsilon \to 0$, we get the conclusion.
\end{proof}

At the end of this section, we complete the proof of \Cref{thm:existence_result_of_continuous_case}.
\begin{proof}[Proof of \Cref{thm:existence_result_of_continuous_case}]
 The uniqueness follows directly from \Cref{thm:comparison_result_of_continuous_case}. We only need to prove the existence of the solution. Let $\{\varepsilon_i\}$ be a sequence of positive numbers converging to $0$ as $i \to \infty$. Consider graphs $G_{\varepsilon_i} = (\overline{\Omega}, E)$ and following discrete infinity Laplace equations
 \begin{align*}
  \begin{cases}
   -\Delta_{\infty}^{\varepsilon_i}u_i(x) = 0, \quad &x \in \Omega, \\
   u_i(x) = g(x), &x \in \partial \Omega.
  \end{cases}
 \end{align*}
 For each $\varepsilon_i$, by \Cref{thm:main_theorem_1}, there exists a unique solution $u_i$ satisfying
 \[
  \limsup_{r \to \infty}\frac{\sup\limits_{d(y, \partial \Omega) \leq r}|u_i(y)|}{r} = 0,
 \]
 where $d$ is the distance on $G_{\varepsilon_i}$. Then on can follows the proof of \cite[Theorem 1.3]{han2025discrete} to get that there exist a uniformly continuous function $u \in C(\overline{\Omega})$ and a subsequence of $\{u_i\}$ such that $u_i$ converges to $u$ locally uniformly. Note that the proof of \cite[Theorem 2.11]{armstrong2012finite} can easily be adapted to our setting. Thus, by \cite[Theorem 2.11]{armstrong2012finite}, $u$ is a solution of equation \eqref{equ:homogeneous_equation_on_unbounded_domain} satisfying the sublinear condition. 
\end{proof}

\section{Inhomogeneous equations on trees}\label{sec4}
In this section, we study inhomogeneous equations on trees. Let $T=(V,E)$ be a tree, where $V=\overline{U}=U \sqcup \delta U$. We deem $\delta U$ as the set of roots of $T$ in this section. We use the following notions.
\begin{itemize}
    \item $|x| = d(x, \delta U)$.
    \item $x^{par}:=\{y\sim x:\ |y|=|x|-1\}$, which is the set of parents of $x$. Specially if $x$ has exactly one parent, $x^{par}$ also denotes the unique vertex.
    \item $x^{chd}:=\{y\sim x:\ |y|=|x|+1\}$, which is the set of children of $x$. 
    \item We say a path $P:\  x_1\sim x_2 \sim x_3 \sim \cdots$ is \emph{downward} if $|x_{i+1}|=|x_i|+1$.
    \item If $P:\ x_1\sim x_2 \sim x_3 \sim \cdots$ is a path and $f\in C(V)$, we write $\sum\limits_{P}f = \sum\limits_{i}f(x_i)$. 
    \item Let $\mathcal{P}$ be the set of downward paths and $\mathcal{P}_{x}$ be the set of downward paths starting from $x$.
\end{itemize}
Now given bounded functions $f\in C(U)$ with $f \geq 0$ and $g\in C(\delta U)$, we study the existence and uniqueness of sublinear solutions to the following equation
\begin{align}\label{eq:inhomogeneous_equation_on_a_tree}
    \begin{cases}
        \Delta_{\infty}u(x)=f(x), \quad &x\in U,\\
        u(x) = g(x), & x\in \delta U.
    \end{cases}
\end{align}

We start from the case that there is exactly one vertex $\bar{x}$ in $\delta U$, in which case we deem $T$ as a rooted tree with root $\bar{x}$. Note that for this case, $x \in U$ has exactly one parent.
\begin{lemma}\label{lem:existence_and_uniqueness_result_for_inhomogeneous_equations_on_a_rooted_tree}
    Suppose that $f\in C(U)$ is nonnegative. If $u$ is a solution of the following equation
    \begin{align}\label{eq:inhomogeneous_equation_on_a_rooted_tree}
        \begin{cases}
            \Delta_{\infty}u(x)=f(x), \ x\in U = V\setminus\{\bar{x}\},\\
            u(\bar{x})=0,
        \end{cases}
    \end{align} 
    and satisfies
    \begin{equation}\label{equ:sublinearcondition}
        \limsup\limits_{r \to \infty} \frac{\sup\limits_{y \in B_r(\bar{x})}|u(y)|}{r}=0,
    \end{equation}
    then for any $x \in U$,
    \begin{equation}\label{equ:sublinearsolution}
     u(x) = u(x^{par}) - \sup_{P \in \mathcal{P}_x}\sum_{P}f.
    \end{equation}
\end{lemma}
\begin{proof}
    Let $u$ be a solution satisfying \eqref{equ:sublinearcondition}. For any $x\in U$, we only need to prove 
    \begin{itemize}
        \item[(i)] $u(x)\leq u(x^{par})$, which implies that $u(x^{par}) = \sup\limits_{y\sim x}u(y)$; 
        \item[(ii)] $\lim\limits_{i\to \infty}(u(x_{i+1})-u(x_i))=0$ for any downward path $P: x_0\sim x_1 \sim x_2 \sim \cdots$;
        \item[(iii)] $u(x) \leq u(x^{par})-\sup\limits_{P\in\mathcal{P}_{x}}\sum\limits_{P}f$;
        \item[(iv)] $u(x) \geq u(x^{par})-\sup\limits_{P\in\mathcal{P}_{x}}\sum\limits_{P}f$.
    \end{itemize}

    Firstly, we prove (i) by contradiction. Suppose that there exists $x \in U$ such that $u(x)-u(x^{par})= \delta >0$. By choosing $x_0=x^{par}, x_1 = x$ and suitable $x_2,x_3,\cdots$, we get a path
    $$P: x_0\sim x_1 \sim x_2 \sim \cdots$$ 
    such that $u(x_{i+1}) \geq \sup\limits_{y\sim x_{i}}u(y) - \frac{\delta}{2^{i+1}}$ for any $i\geq 1$. By the definition of $\Delta_{\infty}$, we have
    \begin{align*}
        0 \leq f(x_i) = \Delta_{\infty}u(x_i) 
        = &\sup_{y\sim x_i}u(y) + \inf_{y\sim x_i}u(y)-2u(x_i)\\
        \leq & [u(x_{i+1}) -u(x_i)] - [u(x_i)-u(x_{i-1})] + \frac{\delta}{2^{i+1}}.
    \end{align*}
    By summing the above inequality, we get
    \begin{align*}
        u(x_{k+1})-u(x_{k}) \geq u(x_1)-u(x_0) - \sum_{i=1}^{k}\frac{\delta}{2^{k+1}} \geq \frac{\delta}{2}, \ \forall \ k\geq 1.
    \end{align*}
    Since $u$ is sublinear, there exists a smallest $K<+\infty$ such that $|x_{K+1}| = |x_K|-1$. Since $T$ is a rooted tree, $x_{K+1} = x_{K-1}$. Then
    \begin{align*}
     u(x_{K+1}) &\geq u(x_K) + \frac{\delta}{2} \\
     &\geq u(x_{K-1}) + \delta = u(x_{K+1}) + \delta,
    \end{align*}
    which is impossible.

    Next we prove (ii) by contradiction. Suppose $\liminf\limits_{i\to\infty}(u(x_{i+1})-u(x_i)) = -b <0$. For any $k$ there exists a $K>k$ such that $u(x_{K+1}) -u(x_K) \leq \frac{-b}{2}$. Note that $\sup\limits_{y\sim x_i}u(y) = u(x_{i-1})$ for any $i\ge 1$ by (i). Then we have
    \begin{align*}
        0 \leq f(x_i) = \Delta_{\infty}u(x_i) 
        = & u(x_{i-1}) + \inf_{y\sim x_i}u(y)-2u(x_i)\\
        \leq & [u(x_{i+1}) -u(x_i)] - [u(x_{i})-u(x_{i-1})],
    \end{align*}
    which implies $u(x_{i+1}) -u(x_i) \leq \frac{-b}{2}$ for any $i\leq K$. Then $u(x_{k+1}) \leq u(x_0) - \frac{b(k+1)}{2}$, which is impossible since $u$ is sublinear.
    
    Now we prove (iii). Let $x_0=x^{par}, x_1 = x$. For any downward path 
    $$P:\ x_0\sim x_1 \sim x_2 \sim \cdots, $$ 
    we have 
    \begin{align*}
        f(x_i) = \Delta_{\infty}u(x_i) 
        = &u(x_{i-1}) + \inf_{y\sim x_i}u(y)-2u(x_i)\\
        \leq & [u(x_{i-1}) -u(x_i)] - [u(x_i)-u(x_{i+1})].
    \end{align*}
    By (ii), we can sum the above inequality, and obtain
    \begin{align*}
        \sum_{i=1}f(x_i) \leq u(x^{par})-u(x).
    \end{align*}
    This proves (iii).

    Finally, we prove (iv). We only need to prove the result under the assumption that $\sup\limits_{P\in\mathcal{P}_{x}}\sum\limits_{P}f < +\infty$. Let $x_0=x^{par}, x_1 = x$. Choose a downward path
    \begin{align*}
        P:\ x_0\sim x_1 \sim x_2 \sim \cdots
    \end{align*}
    such that $u(x_i) \leq \inf\limits_{y\sim x_{i-1}}u(y)+\frac{\varepsilon}{2^i}$, where $\varepsilon >0$. Then 
    \begin{align*}
        f(x_i) = \Delta_{\infty}u(x_i) 
        = &u(x_{i-1}) + \inf_{y\sim x_i}u(y)-2u(x_i)\\
        \geq & [u(x_{i-1}) -u(x_i)] - [u(x_i)-u(x_{i+1})]-\frac{\varepsilon}{2^{i+1}}.
    \end{align*}
    By summing the above inequality, we get
    \begin{align*}
        \sup_{P\in\mathcal{P}_{x}}\sum_{P}f \geq \sum_{i=1}f(x_i) \geq u(x^{par})-u(x)-\frac{\varepsilon}{2}.
    \end{align*}
    The result follows by letting $\varepsilon \to 0$.
\end{proof}

\begin{remark}\label{rmk:necessray_condition_for_sublinear_solution}
 Lemma \ref{lem:existence_and_uniqueness_result_for_inhomogeneous_equations_on_a_rooted_tree} shows that if $u$ is a sublinear solution of equation \eqref{eq:inhomogeneous_equation_on_a_rooted_tree}, then it satisfies \eqref{equ:sublinearsolution}. This provides a necessary condition for the existence of sublinear solutions to equation \eqref{eq:inhomogeneous_equation_on_a_rooted_tree}:
 \[
  \sup\limits_{P\in\mathcal{P}_{x}}\sum\limits_{P}f < +\infty, \quad \forall \ x \in U.
 \]
 However, it is not a sufficient condition: a function $u$ satisfying \eqref{equ:sublinearsolution} is not necessarily sublinear. See the following Example \ref{ex:counterexample2}.
\end{remark}

\begin{example}\label{ex:counterexample2}
 Consider the tree $T$ shown in Figure \ref{fig:counterexample2}. The tree $T$ is essentially composed of countably many half-lines isomorphic to $\mathbb{Z}_+$, all emanating from the vertex $\bar{x}$. On the $k$-th half-line, we define $f(k) = 1$ and $f(j) = 0$ for all $j \neq k$. That is, on the first half-line, $f(1) = 1$ and $f(j) = 0$ for all $j \neq 1$; on the second half-line, $f(2) = 1$ and $f(j) = 0$ for all $j \neq 2$, and so on. Defining the function $u$ according to equation \eqref{equ:sublinearsolution} with $u(\bar{x})=0$, then it is easy to verify that u is not sublinear.
     \begin{figure}
        \centering
        \begin{tikzpicture}
            \draw[thick, gray] (0,0) -- (3,0) (0,0) -- (2.1,2.1) (0,0) -- (2.1,-2.1);
            \draw[gray, thick, dashed] (2.1,2.1) -- (3.5,3.5) (3,0) -- (5,0) (2.1,-2.1) -- (3.5,-3.5);
            \draw[gray, thick, dashed] (0,0) -- (5, 2.5) (0,0) -- (6,1.5) (0,0) -- (5, -2.5) (0,0) -- (6,-1.5) (0,0) -- (2,-4.5);
            \draw[gray, thick, dashed] (3.5,0.5) arc(20:40:4) (3.5,-0.5) arc(-20:-40:4) (1.5,-2.5) arc(-60:-80:4);
            \node at (-0.3, 0) {$\bar{x}$};
            \foreach \a in {0,1,2,3,4}
                \filldraw (\a,0) circle(2pt);
            \foreach \b in {0.7,1.4,2.1,2.8}
                \filldraw (\b,\b) circle(2pt) (\b,-\b) circle(2pt);
        \end{tikzpicture}
        \caption{A function satisfying \eqref{equ:sublinearsolution} is not necessarily sublinear.}
        \label{fig:counterexample2}
    \end{figure}
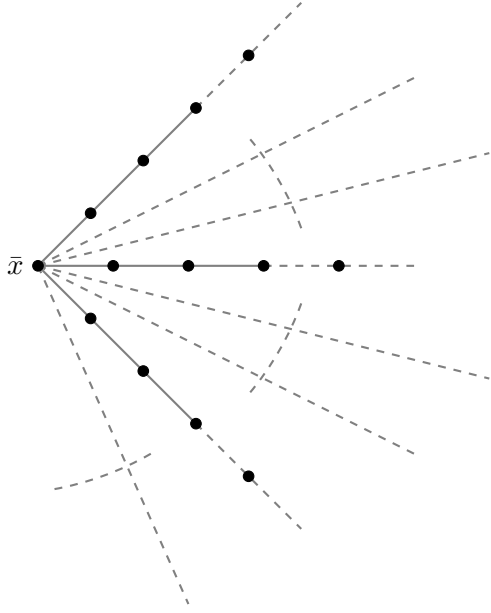
\end{example}

Now we consider the general case. Recall that a tree $T$ has bounded boundary if $\sup\limits_{x,y \in \delta U}d(x, y) < \infty$.

\begin{proof}[Proof of \Cref{thm:inhomogeneous_equations_general_case}]
   We deem $T$ as a tree with multiple roots $\delta U$ and write $$M=\sup\limits_{x,y\in \delta U} d(x,y).$$ If $|x| \geq M+1$, then there exists exactly one $x^{par}\sim x$. Moreover $x^{par}$ lays on all paths connecting $x$ and $\delta U$. If $u$ is a solution satisfying the sublinear condition, then by \Cref{lem:existence_and_uniqueness_result_for_inhomogeneous_equations_on_a_rooted_tree} we have 
    \begin{equation}\label{equ:111111}
        u(x^{par})-u(x) = \sup_{P\in\mathcal{P}_{x}}\sum_{P} f \geq 0,\ \forall \ |x|\geq M+1.
    \end{equation}
    Let $T'$ be the graph induced by $\{x\in V:\ |x| \leq M+1\}$. Note that $\Delta_{\infty}u(x)=2(u(x^{par})-u(x))$ for $|x|=M+1$ on $T'$, then $u$ satisfies the following equation on $T'$:
    \begin{align*}
        \begin{cases}
            \Delta_{\infty}u(x)=f(x), \ &1\leq |x| \leq M,\\
            \Delta_{\infty}u(x)=2\sup\limits_{P\in\mathcal{P}_{x}}\sum\limits_{P} f, &|x| = M+1,\\
            v|_{\delta U}=g.
        \end{cases}
    \end{align*}
    By \Cref{thm:existence_and_uniquness_result_for_inhomogeneous_equations_on_bounded_width_graphs}, the above equation admits a unique solution $u$. By equation \eqref{equ:111111}, we can extend $u$ to the whole $T$, which is the unique solution we want.    
\end{proof}

\section*{Acknowledgments}
Part of this work was completed during the second author's visit to the Max Planck Institute for Mathematics in the Sciences, during which they enjoyed the warm hospitality of the institute.

\bibliographystyle{plain}
\bibliography{bib}

\begin{thebibliography}{10}

\bibitem{armstrong2012finite}
Scott Armstrong and Charles Smart.
\newblock A finite difference approach to the infinity {L}aplace equation and
  tug-of-war games.
\newblock {\em Transactions of the American Mathematical Society},
  364(2):595--636, 2012.

\bibitem{armstrong2010easy}
Scott~N Armstrong and Charles~K Smart.
\newblock An easy proof of jensen's theorem on the uniqueness of infinity
  harmonic functions.
\newblock {\em Calculus of Variations and Partial Differential Equations},
  37:381--384, 2010.

\bibitem{aronsson1967extension}
Gunnar Aronsson.
\newblock Extension of functions satisfying {L}ipschitz conditions.
\newblock {\em Arkiv f{\"o}r Matematik}, 6(6):551--561, 1967.

\bibitem{aronsson2004tour}
Gunnar Aronsson, Michael Crandall, and Petri Juutinen.
\newblock A tour of the theory of absolutely minimizing functions.
\newblock {\em Bulletin of the American mathematical society}, 41(4):439--505,
  2004.

\bibitem{barles2001existence}
Guy Barles and J{\'e}r{\^o}me Busca.
\newblock Existence and comparison results for fully nonlinear degenerate
  elliptic equations without zeroth-order term.
\newblock {\em Communications in Partial Differential Equations},
  26(11-12):2323--2337, 2001.

\bibitem{crandall2007uniqueness}
Michael~G Crandall, Gunnar Gunnarsson, and Peiyong Wang.
\newblock Uniqueness of $\infty$-harmonic functions and the eikonal equation.
\newblock {\em Communications in Partial Differential Equations},
  32(10):1587--1615, 2007.

\bibitem{CIL1992}
Michael~G. Crandall, Hitoshi Ishii, and Pierre-Louis Lions.
\newblock User's guide to viscosity solutions of second order partial
  differential equations.
\newblock {\em Bull. Amer. Math. Soc. (N.S.)}, 27(1):1--67, 1992.

\bibitem{ES2008}
Lawrence~C. Evans and Ovidiu Savin.
\newblock {$C^{1,\alpha}$} regularity for infinity harmonic functions in two
  dimensions.
\newblock {\em Calc. Var. Partial Differential Equations}, 32(3):325--347,
  2008.

\bibitem{G2004}
Thierry Gaspari.
\newblock The infinity {L}aplacian in infinite dimensions.
\newblock {\em Calc. Var. Partial Differential Equations}, 21(3):243--257,
  2004.

\bibitem{han2025discrete}
Fengwen Han and Tao Wang.
\newblock Discrete infinity {L}aplace equations on graphs and tug-of-war games.
\newblock {\em Calculus of Variations and Partial Differential Equations},
  64(2):40, 2025.

\bibitem{HZ2021}
Guanghao Hong and Yizhen Zhao.
\newblock Infinity harmonic functions over exterior domains.
\newblock {\em Int. Math. Res. Not. IMRN}, (14):11093--11102, 2021.

\bibitem{jensen1993uniqueness}
Robert Jensen.
\newblock Uniqueness of {L}ipschitz extensions: minimizing the sup norm of the
  gradient.
\newblock {\em Archive for Rational Mechanics and Analysis}, 123:51--74, 1993.

\bibitem{le2007absolutely}
Erwan Le~Gruyer.
\newblock On absolutely minimizing {L}ipschitz extensions and {PDE}.
\newblock {\em Nonlinear Differential Equations and Applications NoDEA},
  14(1-2):29--55, 2007.

\bibitem{L2014}
Erik Lindgren.
\newblock On the regularity of solutions of the inhomogeneous infinity
  {L}aplace equation.
\newblock {\em Proc. Amer. Math. Soc.}, 142(1):277--288, 2014.

\bibitem{lindqvist2016notes}
Peter Lindqvist.
\newblock {\em Notes on the infinity Laplace equation}.
\newblock Springer, 2016.

\bibitem{LY2012}
Fang Liu and Xiao-Ping Yang.
\newblock Solutions to an inhomogeneous equation involving infinity
  {L}aplacian.
\newblock {\em Nonlinear Anal.}, 75(14):5693--5701, 2012.

\bibitem{LMZ2019}
Guozhen Lu, Qianyun Miao, and Yuan Zhou.
\newblock Viscosity solutions to inhomogeneous {A}ronsson's equations involving
  {H}amiltonians {$\langle A(x)p,p\rangle $}.
\newblock {\em Calc. Var. Partial Differential Equations}, 58(1):Paper No. 8,
  37, 2019.

\bibitem{lu2008inhomogeneous}
Guozhen Lu and Peiyong Wang.
\newblock Inhomogeneous infinity {L}aplace equation.
\newblock {\em Advances in Mathematics}, 217(4):1838--1868, 2008.

\bibitem{lu2008pde}
Guozhen Lu and Peiyong Wang.
\newblock A {PDE} perspective of the normalized infinity {L}aplacian.
\newblock {\em Communications in Partial Differential Equations},
  33(10):1788--1817, 2008.

\bibitem{oberman2005convergent}
Adam Oberman.
\newblock A convergent difference scheme for the infinity {L}aplacian:
  construction of absolutely minimizing {L}ipschitz extensions.
\newblock {\em Mathematics of computation}, 74(251):1217--1230, 2005.

\bibitem{peres2010biased}
Yuval Peres, G{\'a}bor Pete, and Stephanie Somersille.
\newblock Biased tug-of-war, the biased infinity {L}aplacian, and comparison
  with exponential cones.
\newblock {\em Calculus of Variations and Partial Differential Equations},
  38(3-4):541--564, 2010.

\bibitem{PSSW2009}
Yuval Peres, Oded Schramm, Scott Sheffield, and David Wilson.
\newblock Tug-of-war and the infinity {L}aplacian.
\newblock {\em Journal of the American Mathematical Society}, 22(1):167--210,
  2009.

\bibitem{O2005}
Ovidiu Savin.
\newblock {$C^1$} regularity for infinity harmonic functions in two dimensions.
\newblock {\em Arch. Ration. Mech. Anal.}, 176(3):351--361, 2005.

\end{thebibliography}

\end{document}